\documentclass[twoside,11pt]{article}

\usepackage{xcolor}
\usepackage{pgf,tikz}
\usepackage{mathrsfs}
\usetikzlibrary{arrows}
\usepackage{amssymb,amsmath}
\usepackage{amsthm}
\usepackage{hyperref}
\usepackage{mathrsfs}
\usepackage{textcomp}
\usepackage{verbatim}
\usepackage{esint}
\usepackage{sectsty}
\usepackage{cases}
\usepackage{fancyhdr}
\hypersetup{
    colorlinks,%
    citecolor=black,%
    filecolor=black,%
    linkcolor=black,%
    urlcolor=black
}

\usepackage{verbatim}

\usepackage{sectsty}

\makeatletter

\newdimen\bibspace
\renewenvironment{thebibliography}[1]{%
 \section*{\refname 
       \@mkboth{\MakeUppercase\refname}{\MakeUppercase\refname}}%
     \list{\@biblabel{\@arabic\c@enumiv}}%
          {\settowidth\labelwidth{\@biblabel{#1}}%
           \leftmargin\labelwidth
           \advance\leftmargin\labelsep
           \itemsep\bibspace
           \parsep\z@skip     %
           \@openbib@code
           \usecounter{enumiv}%
           \let\p@enumiv\@empty
           \renewcommand\theenumiv{\@arabic\c@enumiv}}%
     \sloppy\clubpenalty4000\widowpenalty4000%
     \sfcode`\.\@m}
    {\def\@noitemerr
      {\@latex@warning{Empty `thebibliography' environment}}%
     \endlist}

\makeatother

\makeatletter

\newtheorem{thm}{Theorem}[section]
\newtheorem{lem}[thm]{Lemma}
\newtheorem{prop}[thm]{Proposition}

\newtheorem{cor}[thm]{Corollary}
\newtheorem{rem}[thm]{Remark}

\def\Xint#1{\mathchoice
  {\XXint\displaystyle\textstyle{#1}}%
  {\XXint\textstyle\scriptstyle{#1}}%
  {\XXint\scriptstyle\scriptscriptstyle{#1}}%
  {\XXint\scriptscriptstyle\scriptscriptstyle{#1}}%
  \!\int}
\def\XXint#1#2#3{{\setbox0=\hbox{$#1{#2#3}{\int}$}
  \vcenter{\hbox{$#2#3$}}\kern-.5\wd0}}

\def\dashint{\Xint-}

\newcommand{\al}{\alpha}                \newcommand{\lda}{\lambda}
\newcommand{\om}{\Omega}                \newcommand{\pa}{\partial}
\newcommand{\va}{\varepsilon}           \newcommand{\ud}{\mathrm{d}}
\newcommand{\be}{\begin{equation}}      \newcommand{\ee}{\end{equation}}
                 
\newcommand{\Lda}{\Lambda}              
\newcommand{\R}{\mathbb{R}}

\begin{document}

\title{\textbf{The critical semilinear elliptic equation with isolated boundary singularities}
\bigskip}

\author{\medskip
Jingang Xiong\footnote{Supported in part by NSFC 11501034, key project of NSFC 11631002 and NSFC 11571019.} }

\date{}

\fancyhead{}
\fancyhead[CO]{\textsc{Isolated boundary singularity}}
\fancyhead[CE]{\textsc{J. Xiong}}

\fancyfoot{}

\fancyfoot[CO, CE]{\thepage}

\renewcommand{\headrule}{}

\maketitle

\begin{abstract} We establish quantitative asymptotic behaviors for nonnegative solutions of the critical semilinear  equation $-\Delta u=u^{\frac{n+2}{n-2}}$ with isolated boundary singularities, where $n\ge 3$ is the dimension.

\end{abstract}

\section{Introduction}

The internal isolated singularity for positive solutions of the semilinear equation $-\Delta u= u^{p}$ has been very well understood, where $\Delta$ is the Laplace operator,  $1<p\le \frac{n+2}{n-2}$ is a parameter and $n\ge 3$ is the dimension. See Lions \cite{Lions} for $1<p<\frac{n}{n-2}$, Gidas-Spruck \cite{GS} for $\frac{n}{n-2}< p<\frac{n+2}{n-2}$, Aviles \cite{A} for $p=\frac{n}{n-2}$, Caffarelli-Gidas-Spruck \cite{CGS} for $\frac{n}{n-2}\le p\le \frac{n+2}{n-2}$ and Korevaar-Mazzeo-Pacard-Schoen \cite{KMPS} for $p= \frac{n+2}{n-2}$. The Sobolev critical exponent $p= \frac{n+2}{n-2}$ case is of particular interest,  because the equation connects to the Yamabe problem and the conformal invariance leads to a richer isolated singularity structure. See also Li \cite{Li06} and Han-Li-Teixeira \cite{HLT} for conformally invariant fully nonlinear elliptic equations.

The Dirichlet boundary isolated singularity for the same equation has also been studied in many cases. Asymptotic behaviors of singular solutions have  been established by Bidaut-V\'eron-Vivier \cite{BVi} for $1<p<\frac{n+1}{n-1}$  and Bidaut-V\'eron-Ponce-V\'eron \cite{BPV, BPV_a} for $\frac{n+1}{n-1}\le p<\frac{n+2}{n-2}$. Existence of singular solutions vanishing on boundaries of bounded domains except finite points has been obtained by del Pino-Musso-Pacard \cite{dMP} for $p<\frac{n+2}{n-2}$.  The exponent $\frac{n+1}{n-1}$ corresponding to $\frac{n}{n-2}$ for the interior singularity was discovered  by Br\'ezis-Turner \cite{BT}. Under a blow up rate assumption Bidaut-V\'eron-Ponce-V\'eron \cite{BPV, BPV_a} obtain refined asymptotic behaviors for the supercritical case $\frac{n+2}{n-2}<p<\frac{n+1}{n-3}$. We refer to \cite{BPV} and references therein for related results on boundary singularity.

This paper is concerned with the remaining critical case: $p=\frac{n+2}{n-2}$. The conformal invariance again produces additional complexity and the boundary condition makes the asymptotic analysis of \cite{CGS} and \cite{KMPS} fail. As said in Bidaut-V\'eron-Ponce-V\'eron \cite{BPV_a}, one can show

\begin{prop}\label{prop:2} Denote $\R^n_+=\{x=(x',x_n)\in \R^n: x_n>0\}$. Let $u \in C^2(\R^n_+)\cap C(\bar \R^n_+\setminus \{0\} )$ be a nonnegative solution of
\be\label{eq:main-1}
\begin{cases}
\begin{aligned}
-\Delta u&=n(n-2) u^{\frac{n+2}{n-2}}  & \quad &\mbox{in }\R^n_+,\\
u &= 0 &\quad &\mbox{on } \pa  \R^n_+ \setminus \{0\}.
\end{aligned}
\end{cases}
\ee
Suppose $0$ is a non-removable singularity of $u$, then $u$ depends only on $|x'|$ and $x_n$, and $\pa_r u(r,x_n)< 0$ for all $r=|x'|>0$.

\end{prop}

Note that nothing about the behavior of $u$ at infinity is assumed in Proposition \ref{prop:2}.

Let $u$ be a solution of \eqref{eq:main-1} and define $U(t,\theta):=|x|^{\frac{n-2}{2}} u(|x|\cdot \theta)$ with $t=-\ln |x|$. Then we have
\be\label{eq:cylinder}
\pa_{tt}^2 U+\Delta_{\mathbb{S}^{n-1}} U-\frac{(n-2)^2}{4} U+n(n-2) U^{\frac{n+2}{n-2}}= 0\quad \mbox{on }\R \times \mathbb{S}^{n-1}_+,
\ee
\be\label{eq:cylinder-1}
U= 0\quad \mbox{on }\R\times \pa  \mathbb{S}^{n-1}_+,
\ee
where $\mathbb{S}^{n-1}_+=\{\theta=(\theta_1,\dots, \theta_n)\in \mathbb{S}^{n-1}:\theta_n>0\}$. By Proposition \ref{prop:2}, $U(t,\theta)=U(t,\theta_n)$. In contrast to the internal singularity studied by
Caffarelli-Gidas-Spruck \cite{CGS} and Korevaar-Mazzeo-Pacard-Schoen \cite{KMPS}, we lose ODE analysis to classify all solutions of equation \eqref{eq:cylinder}-\eqref{eq:cylinder-1}. del Pino-Musso-Pacard \cite{dMP} conjectured that there exists a one-parameter family of solutions of \eqref{eq:cylinder}-\eqref{eq:cylinder-1}.  Bidaut-V\'eron-Ponce-V\'eron \cite{BPV, BPV_a} proved that there exists a unique $t$-independent solution. Existence of $t$-dependent solutions and a priori estimates are left open.

Let $\psi$ be a $C^2$ function in  $\R^{n-1}$ satisfying
\[
\psi(0)=0, \quad\nabla \psi(0)=0.
\]
Let $Q_R= \{x=(x',x_n):x_n >\psi(x')\} \cap B_R$ and $\Gamma_R= \{x=(x',x_n):x_n =\psi(x')\} \cap B_R$, where $B_R$ is the open ball center at $0$ with radius $R$.
We consider nonnegative solutions of
\be\label{eq:main}
\begin{cases}
\begin{aligned}
-\Delta u&=n(n-2) u^{\frac{n+2}{n-2}}  & \quad &\mbox{in }Q_1,\\
u &= 0 &\quad &\mbox{on } \Gamma_1\setminus \{0\}.
\end{aligned}
\end{cases}
\ee

\begin{thm} \label{thm:1} Let $u\in C^2(\bar Q_1 \setminus \{0\})$ be a nonnegative solution of \eqref{eq:main}. Then for each $0<\gamma<1$ there exists a constant $C(\gamma)\ge 1$ such that for all $x\in Q_{1/2}$ with $\frac{d(x)}{|x|}\le \gamma$,
\be\label{5}
u(x)\le C(\gamma) d(x) |x|^{-\frac{n}{2}}
\ee
and
\be\label{7}
u(x',x_n)=\bar u(|x'|,x_n)(1+O(|x|)) \quad \mbox{if }x_n>\frac{1}{\delta}\max_{|y'|=|x'|}| \psi(y')|,
\ee
where $d(x)= dist(x,\Gamma_1)$, $\bar u(x',x_n)=\dashint_{\mathbb{S}^{n-2}}u(|x'|\theta, x_n)\,\ud \theta$, $O(|x|)\le C(\gamma)|x|$ as $x\to 0$, and $\delta>0$ depends only on the $C^2$ norm of $\psi$.

If the above inequality \eqref{5} holds for $\gamma=1$, then either $0$ is a removable singularity or there exists a constant $C>0$ such that
\be\label{6}
u(x)\ge \frac{1}{C}d(x) |x|^{-\frac{n}{2}} \quad \mbox{for all }x\in Q_{1/2}.
\ee
Furthermore, \eqref{7} holds for all $x\in Q_{1/2}$.
\end{thm}

Without assuming \eqref{5} holds up to $\gamma=1$, we are still able to show some sort of asymptotic symmetry for almost all $x\in Q_{1/2}$ close to $0$; see Proposition \ref{prop:movingsphere-1}.

The second conclusion of Theorem \ref{thm:1} is a partial answer of a question of \cite{BPV_a} (see Remark 1 of \cite{BPV_a}).

The method of proof of Theorem \ref{thm:1} can be adapted to study boundary singularity of critical equations with nonlinear Neumann boundary conditions, which we leave to another paper.
Motivated by conformal geometry, boundary singularity of linear (degenerate) elliptic equation with a critical nonlinear Neumann boundary condition has been studied in Caffarelli-Jin-Sire-Xiong \cite{CJSX}, Jin-de Queiroz-Sire-Xiong \cite{JQSX}, and Sun-Xiong \cite{SX}.

The organization of paper and crucial steps of the proofs are as follows. In section \ref{sec:pre}, we recall some basic facts of elliptic equations with zero Dirichlet condition, such as boundary Harnack inequality, a special maximum principle and a boundary version of B\^ocher theorem. In section \ref{sec:partialbound}, we prove the partial upper bound in the main theorem. A classical result of  Berestycki-Nirenberg \cite{BN} plays an important role.  In section \ref{sec:lowerbound}, we establish the lower bound. The Pohozaev identity and B\^ocher theorem are used crucially. In section \ref{sec:symmetry}, we prove symmetry results via the moving spheres method developed by Li-Zhu \cite{LZhu}; see also Li-Zhang \cite{LZhang}. In particular, we present a proof of Proposition \ref{prop:2} by this method. Here the boundary Harnack inequality is used repeatedly. Our proof of Proposition \ref{prop:2} can be adapted to give another proof of a result of Dancer \cite{D}.

\bigskip

\noindent\textbf{Acknowledgments:}
We would like to thank the referee for her/his invaluable suggestions.

\medskip

\section{Preliminaries}
\label{sec:pre}

Since we will work in neighborhoods of partial boundaries of domains, we let $\psi$ be a $C^2$ function in  $\R^{n-1}$ satisfying
\[
\psi(0)=0, \quad\nabla \psi(0)=0
\]
and denote $Q_R= \{x=(x',x_n):x_n >\psi(x')\} \cap B_R$, $\Gamma_R= \{x=(x',x_n):x_n =\psi(x')\} \cap B_R$ and $d(x)=dist(x,\Gamma_1)$,  where $B_R$ is the open ball center at $0$ with radius $R$. We denote $B_R^+(y)$ for $y=(y',0)$ as $\{x\in B_R(y):x_n>0\}$, $\pa' B_R^+(y):=B_R(y) \cap \pa \R^n_+$ and $\pa'' B_R^+(y):=B_R(y)\cap \R^n_+$. We assume
\[
\|\psi\|_{C^2(\pa' B_1^+)} \le A_1.
\]

\begin{lem}\label{lem:boundary-nonsing} Let $ u\in C^2(\bar Q_1)$ be a nonnegative solution of
\[
\begin{cases}
\begin{aligned}
-\Delta u&=a(x) u  & \quad &\mbox{in }Q_1,\\
u &= 0 &\quad &\mbox{on } \Gamma_1,
\end{aligned}
\end{cases}
\]
where $a(x)\in L^\infty( Q_1)$. Then
\[
 \frac{u(x)}{d(x)}\le C_0 \frac{u(y)}{d(y)} \quad \mbox{for }x,y\in Q_{1/2}
\]
and
\[
\left| \frac{u(x)}{d(x)}-\frac{u(y)}{d(y)}\right| \le C_0|x-y|^{\al} \quad \mbox{for }x,y\in Q_{1/2},
\]
where $C_0\ge 1$ and $\al\in (0,1)$ are constants depending only on $n$, $A_1$, $\|a\|_{L^\infty (Q_{1})}$ and $\|u\|_{L^\infty(Q_1)}$.

\end{lem}

The above result is a simple version of the boundary Harnack inequality; see  Caffarelli-Fabes-Mortola-Salsa \cite{CF+} or Krylov \cite{Kr}. If only a differential inequality is assumed, we have

\begin{lem}\label{lem:boundary-sing}Let $ u\in C^2(\bar Q_1\setminus \{0\})$ be a nonnegative solution of
\[
\begin{cases}
\begin{aligned}
-\Delta u&\ge 0  & \quad &\mbox{in }Q_1,\\
u &= 0 &\quad &\mbox{on } \Gamma_1\setminus \{0\},
\end{aligned}
\end{cases}
\]
If $u>0$ somewhere in $Q_1$, then there exists a  constant $c_0>1$ such that
\[
u(x)\ge  d(x)/c_0 \quad \mbox{for all }x\in Q_{1/2}.
\]
\end{lem}

\begin{proof} It is easy to check that there exist constants $0<\delta<1/4$ and $A\ge 1$ with $A\delta<1/2$ such that
 \be  \label{eq:2.1}
 \Delta(\frac{d(x)}{1-Ad(x)})\ge 0 \quad \mbox{for }x\in Q_{1/2}, ~d(x)<\delta.
 \ee

It follows from the strong maximum principle that $u>0$ in $Q_1$. By Hopf Lemma and the compactness of $\Gamma_1\cap \pa B_{1/2}$, there exists a constant $b_0>0$ such that
\[
-\pa_{\nu} u\ge b_0>0 \quad \mbox{on } \Gamma_1\cap \pa B_{1/2}.
\] By the continuity of $\nabla u$ on $Q_1\cap  \pa B_{1/2}$, one can find a constant $C>0$ such that
\[
u(x)\ge \frac{1}{C} \frac{d(x)}{1-Ad(x)} \quad \mbox{on } \pa (Q_{1/2}\cap \{d(x)<\delta\}).
\]  In view of \eqref{eq:2.1}, by maximum principle we have
\[
u(x)\ge \frac{1}{C} \frac{d(x)}{1-Ad(x)}  \quad \mbox{in } Q_{1/2}\cap \{d(x)<\delta\}. \]
The lemma follows immediately.

\end{proof}

We will also use a maximum principle when the coefficient of zero order term without sign restriction but being small.

\begin{lem} \label{lem:mp} Let $u\in C^2(\bar Q_1\setminus \{0\})$ be a solution of
\[
\Delta u +a u \le 0 \quad \mbox{in } Q_1, \quad u=0 \quad \mbox{on }\Gamma_1\setminus \{0\}.
\]
There exists a small constant $\delta=\delta(n)>0$ such that if $|a(x)|\le \delta |x|^{-2}$, and $u\ge 0$ on $\pa (Q_1\setminus Q_\va)$ for some $\va>0$, there holds
\[
u\ge 0 \quad \mbox{in } Q_1 \setminus Q_\va.
\]

\end{lem}

\begin{proof} Multiplying both sides by $u^-=\max\{-u,0\}$ and using the Hardy inequality, the lemma follows immediately.

\end{proof}

\begin{lem}\label{lem:exterior} For $R>0$,  let $u\in C^2(\R^n_+\setminus B_{R})$ and be continuous up to the boundary $\pa \R^n_+\setminus B_R$ except a point $\bar x= (\bar x',0)$ with $|\bar x'|>R$. Suppose that
\[
-\Delta u\ge 0\quad \mbox{in }\R^n_+\setminus B_R^+, \quad u(x',0)=0 \quad \mbox{for }x'\neq \bar x',
\] and $u\ge 0$.
Then
\[
u(x)\ge \frac{R^n x_n}{|x|^n} \inf_{y\in\pa''B_R^+ } \frac{u(y)}{y_n} \quad \mbox{in }\R^n_+\setminus B_R^+.
\]

\end{lem}

\begin{proof} For $\va>0$, let
\[
\phi_\va(x)=  \frac{R^n x_n}{|x|^n} \inf_{y\in\pa''B_R^+ } \frac{u(y)}{y_n}-\va.
\]
By maximum principle we have $u\ge \phi_\va$. Sending $\va\to 0$, the lemma follows.

\end{proof}

 To establish the lower bound in Theorem \ref{thm:1}, we need a well-known boundary B\^ocher type theorem. See Marcus-V\'eron \cite{MV} for a nonlinear version.

\begin{lem}\label{lem:bocher} Let $u\in C^2(B_1^+)\cap C^0(\bar B_1^+\setminus \{0\})$ be a nonnegative solution of
\[
-\Delta u=0 \quad \mbox{in }B_1^+, \quad u(x',0)=0\quad \mbox{for }x'\neq 0.
\]
Then
\[
u(x)= a \frac{x_n}{|x|^n} +h(x),
\]
where $a\ge 0$ is a constant and
\[
-\Delta h=0 \quad \mbox{in }B_1^+, \quad h(x',0)=0\quad \mbox{for }|x'|<1.
\]
Furthermore, if $B_1^+$ is replaced by $\R^n_+$, then $h=bx_n$ for some nonnegative constant $b$.

\end{lem}

\section{A partial upper bound}
\label{sec:partialbound}

The following lemma is  an easy consequence of a classical result of Berestycki-Nirenberg \cite{BN}.

\begin{lem} \label{lem:crucial} Let  $\om$ be a bounded domain in $\R^n$, which is convex
in the $x_1$ direction and symmetric with respect to the hyperplane $\{x_1=0\}$. Let $u\in C^2(\om)$ be a positive solution of
\[
-\Delta u= f(u) \quad \mbox{in }\om \quad \mbox{and}\quad  u=0 \quad \mbox{on }\pa \om \cap \{x_1>0\},
\]
where $f:[0,\infty)\to [0,\infty)$ is local Lipschitz continuous.
Then $\pa_{x_1} u(x)<0$ for $x\in \om$ with $x_1>0$.

\end{lem}

In order to use Lemma \ref{lem:crucial}, let us define the conformal transform $F:\R^n_+ \to B_1$ by
\[
F(x):= \left(\frac{2x'}{|x'|^2+ (1+x_n)^2}, \frac{|x|^2-1}{|x'|^2+(1+x_n)^2}\right).
\]
Then $F\Big|_{\pa \R^n_+}$ is the inverse of stereographic projection.
Let $u$ be a solution of \eqref{eq:main}, then the function
\[
v(z):= |J_{F}|^{-\frac{n-2}{2n}} u(x)
\]
satisfies
\[
-\Delta v= v^{\frac{n+2}{n-2}} \quad \mbox{in }F(Q_1\cap B_1^+),
\]
where  $z:=F(x)$ and $|J_{F}|$ is the Jacobian determinant of $F$.

By performing a Kelvin transform with respect to a sphere of small radius below $Q_1$ and re-labeling coordinates, we may assume that $\psi $ is convex. Since $\psi(0)=0$ and $\nabla \psi(0)=0$, we have $\psi\ge 0$.  For $0<r<1$, denote
\[ \tilde Q_{1/r}= F( \frac{1}{r} Q_1), \quad  \tilde \Gamma_{1/r}=F(\frac{1}{r} \Gamma_1)
\]
and
\[
v(F(x))=|J_{F}|^{-\frac{n-2}{2n}} r^{\frac{n-2}{2}} u(r x).
\]
Then
\be
-\Delta v=n(n-2) v^{\frac{n+2}{n-2}} \quad \mbox{in }\tilde Q_{1/r}, \quad  v=0\quad \mbox{on }\tilde \Gamma_{1/r}\setminus \{-e_n\},
\ee
where $e_n= (0',1)$. Note that $rx\in Q_1$ implies $x\in \{y_n>\frac{1}{r}\psi(ry')\}$. Clearly,
\begin{align}
(z',z_n):=& \left(\frac{2x'}{|x'|^2+(1+r^{-1}\psi(rx'))^2}, \frac{|x'|^2+r^{-2}\psi(rx')^2-1}{|x'|^2+(1+r^{-1}\psi(rx'))^2}\right) \nonumber \\&
\to  \left(\frac{2x'}{|x'|^2+1}, \frac{|x'|^2-1}{|x'|^2+1}\right) \quad \mbox{in }C_{loc}^2(\R^{n-1}) \mbox{ as }r\to 0 \label{eq:c2tosphere}
\end{align}
and thus
\be\label{eq:tosphere}
\tilde \Gamma_{1/r}\to \pa B_1\setminus \{e_n\}.
\ee

Now we able to prove a partial upper bound.

\begin{prop}\label{lem:2.1} Let $u\in C^2(\bar Q_1 \setminus \{0\})$ be a nonnegative solution of \eqref{eq:main}. Then for every $0<\gamma<1$, there exists a constant $C(\gamma)>0$ such that
\[
 |x|^{\frac{n-2}{2}} u(x) \le C(\gamma) \quad \forall ~ x\in Q_{3/4} \mbox{ with }\frac{d(x)}{|x|} <\gamma.
\]

\end{prop}

\begin{proof}
By \eqref{eq:c2tosphere} and \eqref{eq:tosphere},  for any constant $0<\delta<1/2$ there exists a constant $r_0>0$  such that $\tilde Q_{1/r} \cap \{z_n<1-\delta\}$ is a convex body for all $0<r<r_0$.
Together with Lemma \ref{lem:crucial}, there exits a constant $c(\delta)>\delta$ with $\lim_{\delta \to 0}c(\delta)=0$, depending only on $n$  and $\delta$, such that
$v$ has no critical points in the region $\tilde Q_{1/r}\setminus \mathcal{C}_{c(\delta)}$, where $\mathcal{C}_{c(\delta)}$ is the cone generated by the vertex $-e_n$ and $B_{c(\delta)}(e_n)$. Choose $\delta$ small such that
\[
F(\{x\in Q_{1/4}: \frac{x_n}{|x|}\le \gamma\})\cap \mathcal{C}_{2c(\delta)}=\emptyset.
\]
Since $F$ is a conformal map,
\be \label{eq:equi-1}
dist(F(x), -e_n)=  2|x| \cdot\frac{\sqrt{|x|^2+2x_n+1}}{|x'|^2+ (1+x_n)^2} \thickapprox |x|
\ee
and $|J_{F}(x)|$ is smooth positive smooth function on $\bar B_{1/2}^+$, we only need show that for $0<r<r_0$
\be \label{eq:reduction1}
\limsup_{\tilde Q_{1/r}\setminus \mathcal{C}_{2c(\delta)} \ni z\to -e_n} |z|^{\frac{n-2}{2}}v(z)<\infty.
\ee

 Suppose the contrary that there exists a sequence $\{z_j\}_{j=1}^\infty \subset \tilde Q_{1/r}\setminus \mathcal{C}_{2c(\delta)} $ such that
\[
z_j\to -e_n\quad \mbox{as } j\to \infty,
\]
and
\be\label{eq:cl1}
|z_j|^{\frac{n-2}{2}}v(z_j)\to \infty\quad \mbox{as }j\to \infty.
\ee

Let $\theta=\arcsin c(\delta)- \arcsin \frac{c(\delta)}{2}>0$ be the cone angle error between $\mathcal{C}_{2c(\delta)}$ and $ \mathcal{C}_{c(\delta)}$. It is easy to see that
\be \label{eq:angle}
B_{|z_j|\sin \theta}(z_j)\cap  \mathcal{C}_{c(\delta)} =\emptyset.
\ee
Consider
\[
v_j(z):=\left(\frac{ \sin \theta }{2} |z_j| -|z-z_j|\right)^{\frac{n-2}{2}} v(z),\quad |z-z_j|\leq  \frac{\sin \theta }{2} |z_j| .
\]
Let $|\bar z_j-z_j|< \frac{\sin \theta }{2} |z_j| $ satisfy
\[
v_j(\bar z_j)=\max_{|z-z_j|\leq  \frac{\sin \theta }{2} |z_j| }v_j(z),
\]
and let
\[
2\mu_j:= \frac{\sin \theta }{2} |z_j| -|\bar z_j-z_j|.
\]
Then
\be \label{eq:cl2}
0<2\mu_j\leq  \frac{\sin \theta }{2} |z_j| \quad\mbox{and}\quad  \frac{\sin \theta }{2} |z_j| -|z-z_j|\ge\mu_j \quad \forall ~ |z-\bar z_j|\leq \mu_j.
\ee
By the definition of $v_j$, we have
\be \label{eq:cl3}
(2\mu_j)^{\frac{n-2}{2}}v(\bar z_j)=v_j(\bar x)\ge v_j(z)\ge (\mu_j)^{\frac{n-2}{2}}v(z)\quad \forall ~ |z-\bar z_j|\leq \mu_j.
\ee
Thus, we have
\[
2^{\frac{n-2}{2}}v(\bar z_j)\ge v(z)\quad \forall ~ |z-\bar z_j|\leq \mu_j.
\]
We also have
\be\label{eq:cl4}
(2\mu_j)^{\frac{n-2}{2}}v(\bar z_j)=v_j(\bar z_j)\ge v(z_j)= \left(\frac{\sin \theta |z_j|}{2}\right)^{\frac{n-2}{2}}v(z_j)\to \infty\quad\mbox{as }j\to\infty.
\ee
Now, consider
\[
w_j(y)=\frac{1}{v(\bar z_j)}v(\bar z_j+\frac{y}{v(\bar z_j)^{\frac{2}{n-2}}}), \quad y\in \om_j,
\]
where
\[
\om_j:=\left\{y\in \R^n| \bar z_j+\frac{y}{v(\bar z_j)^{\frac{2}{n-2}}}\in \tilde  Q_{1/r}\right\}.
\]
Then $w_j$ satisfies $w(0)=1$ and
\be \label{eq:ext1}
-\Delta w_j=  w_j^{\frac{n+2}{n-2}} \quad \mbox{in }\om_j
\ee
Moreover, it follows from \eqref{eq:cl3} and \eqref{eq:cl4} that
\[
w_j(y)\leq 2^{\frac{n-2}{2}} \quad\mbox{in } B_{R_j}\cap \om_j,
\]
where \[R_j:=\mu_j v(\bar z_j)^{\frac{2}{n-2}}\to \infty \mbox{ as } j\to \infty.\]

Since $\psi$ is $C^2$ and $\nabla \psi(0)=0$, we have two possibilities:

If $B_{R_j}\cap \om_j \to \R^n\cap \{x_n> -\rho\}$ for some $\rho>0$, then by the up to boundary estimates for 2nd order linear elliptic equations there exists a subsequence of $\{w_j\}$, which is  still denoted as $\{w_j\}$, satisfying
\[
w_j \to w \quad \mbox{in } C^2_{loc}(\R^n\cap \{y_n\ge \rho\})
\]
for some $w$ satisfying
\be\label{eq:halfspace}
-\Delta w=n(n-2) w^{\frac{n+2}{n-2}} \quad \mbox{in }\R^n\cap \{y_n\ge \rho\}, \quad w(y',\rho)=0.
\ee
By \cite{D}, $w=0$. This is impossible since $w(0)=1$.

If $B_{R_j}\cap \om_j \to \R^n$, then by the interior estimates for 2nd order linear elliptic equations there exists a subsequence of $\{w_j\}$, which is  still denoted as $\{w_j\}$, satisfying
\[
w_j \to w \quad \mbox{in } C^2_{loc}(\R^n)
\]
for some $w$ satisfying
\be\label{eq:wholespace}
-\Delta w=n(n-2) w^{\frac{n+2}{n-2}} \quad \mbox{in }\R^n.
\ee
By \cite{CGS} we have
\be\label{eq:limit}
w(y)=(\frac{\lda}{1+\lda^2|y-\bar y|^2})^{\frac{n-2}{2}}
\ee
for some point $\bar y\in \R^n$ and $1\le \lda\le 2^{\frac{n-2}{2}}$.

Since $w_j \to w$ in $C^2_{loc}(\R^n)$ and $\nabla w(\bar y)=0$ and $\nabla^2 w(\bar x)$ is negative definite, for large $j$ there exists $y_j\in B_1(\bar y)$ such that $\nabla w_j(y_j)=0$. By the definition of $w_j$, we have
\[
\nabla v(\bar z_j+ \frac{y_j}{ u(\bar z_j)^{\frac{2}{n-2}}})=0.
\]
This is impossible, since $\bar z_j+ \frac{y_j}{ u(\bar z_j)^{\frac{2}{n-2}}} \in \tilde Q_{1/r}\setminus \mathcal{C}_{c(\delta)}$ where $v$ does not have any critical point.
Therefore, \eqref{eq:reduction1} holds. Scaling back to $u$ and using \eqref{eq:equi-1}, we proved Proposition \ref{lem:2.1}.

\end{proof}

\begin{cor}\label{cor:partial bound} Assume the assumptions in Proposition \ref{lem:2.1}. Let $0<\gamma<1$.  Then for $x\in Q_{1/4}$ with $\frac{d(x)}{|x|} \le \gamma$,
\be \label{eq:ub-1}
 u(x)\le C(\gamma) d(x) |x|^{-\frac{n}{2}},
\ee
\be \label{eq:ub-2}
|\nabla ^k  u|\le C(\gamma) |x|^{-\frac{n-2}{2}-k}, \quad k=0,1,2,
\ee  and for any $0<s<\frac{1}{4}$, $x,y\in Q_{2s}\setminus Q_{s/2} \mbox{ with }\frac{d(x)}{|x|}, \frac{d(y)}{|y|}\le \gamma,$
\be \label{eq:harnack}
\frac{u(x)}{d(x)} \le C(\gamma) \frac{u(y)}{d(y)},
\ee
where $r=|x|$, $\theta= \frac{x}{|x|}$, $C(\gamma)>0$ depends on $\gamma$ but not $s$.
\end{cor}
\begin{proof} For $0<s<\frac{1}{4}$, let
\[
v_s(x)=s^{\frac{n-2}{2}} u(s x).
\]
Denote $\tilde Q_r:=\{x: sx\in Q_{rs}\}$, $\tilde \Gamma_r:= \{x: sx\in \Gamma_{rs}\}$ for any $r>0$, and $\tilde d(x)=dist(x,\tilde \Gamma_1)$.
Then
\[
-\Delta v_s=n(n-2) v_s^{\frac{n+2}{n-2}} \quad \mbox{in }\tilde Q_{3}\setminus\tilde Q_{1/4}, \quad v_s=0 \quad \mbox{on }\Gamma_{3}\setminus \Gamma_{1/4}.
\]
It follows from Proposition \ref{lem:2.1} that $v_s(x)\le C(\gamma')$ for $x\in \tilde Q_{3}\setminus\tilde Q_{1/4}$ with $\frac{\tilde d(x)}{|x|}<\gamma'$, where $0<\gamma<\gamma'<1$. By Lemma \ref{lem:boundary-nonsing} and the standard linear elliptic equations theory, for  $ x,y\in \tilde Q_2\setminus Q_{1/2}$ with $\frac{\tilde d(x)}{|x|}\le \gamma$ and $\frac{\tilde d(y)}{|y|}\le \gamma$ we have
\[
v_s(x)\le \tilde  d(x),
\]
\[
 |\nabla^k v_s(x)|\le C, \quad k=1,2,
\]
\[
\frac{v_s(x)}{\tilde d(x)}\le C \frac{v_s(y)}{\tilde d(y)},
\]
where $C$ depends only on $n, C(\gamma ')$. Scaling back to $u$, the above three inequalities yield \eqref{eq:ub-1}, \eqref{eq:ub-2} and \eqref{eq:harnack}, respectively.

Therefore, we complete the proof.

\end{proof}

\begin{rem}\label{rem:full} If $u(x)\le C|x|^{\frac{2-n}{2}}$ for all $x\in Q_{1/4}$, Corollary \ref{cor:partial bound} holds for $\gamma=1$.

\end{rem}

\section{A lower bound and removability}
\label{sec:lowerbound}

\begin{lem}[Pohozaev identity] \label{lem:pohozaev} Let $u\in C^2(\bar Q_1 \setminus \{0\})$ be a nonnegative solution of \eqref{eq:main}.  Then for all $0<r<1$ there holds
\[
 P(u,r):= \int_{Q_1\cap \pa B_r} \frac{n-2}{2} u\frac{\pa u}{\pa r}-\frac{r}{2} |\nabla u|^2+r|\frac{\pa u}{\pa r}|^2+\frac{(n-2)^2}{8} r u^{\frac{2n}{n-2}}\, \ud S=c_0,
\]
where $c_0$ is constant independent of $r$.

\end{lem}

The proof of Lemma \ref{lem:pohozaev} is standard by now. $ P(u,r)$ is called \textit{Pohozaev integral} sometimes in the literature.

\begin{prop}\label{prop:lowbound} Let $u\in C^2(\bar Q_1 \setminus \{0\})$ be a nonnegative solution of \eqref{eq:main}. If
\be
\limsup_{Q_1\ni x\to 0} |x|^{\frac{n-2}{2}} u(x) <\infty
\ee
and
\[
\liminf_{Q_1\ni x\to 0} d(x)^{-1} |x|^{\frac{n}{2}} u(x)=0,
\]
then
\[
\lim_{Q_1\ni x\to 0} d(x)^{-1} |x|^{\frac{n}{2}} u(x)=0.
\]

\end{prop}

\begin{proof} Suppose the contrary that
\[
\limsup_{Q_1\ni x\to 0} d(x)^{-1} |x|^{\frac{n}{2}} u(x) =C_0>0.
\]
Since $\liminf_{Q_1\ni x\to 0} d(x)^{-1} |x|^{\frac{n}{2}} u(x)=0$, by the Harnack inequality \eqref{eq:harnack} in annulus  we can find sequences  $ x_j=(0',(x_j)_n)\to 0$ and $y_j=(0',(y_j)_n)\to 0$ as $j\to \infty$ satisfying
\[
|x_j|^{\frac{n-2}{2}} u(x_j)\to 0 \quad \mbox{ and } \quad |y_j|^{\frac{n-2}{2}} u(y_j)\to C_0^* \quad \mbox{as }j\to \infty,
\]
where $0<\bar C_0^* \le C_0$. In view of this oscillation picture, without loss of generality we assume $(x_j)_n$ are local minimum of $x_n^{-1} |x|^{\frac{n}{2}} u(x)$ restricted to the line $(0',x_n)$. It follows that
\be \label{eq:criticalpoint}
\frac{\pa }{\pa r}(x_n^{-1} |x|^{\frac{n}{2}} u(x) )\Big|_{x=x_j}=0.
\ee
Let $r_j=|x_j|=(x_j)_n>0$, and
\[
w_j(x)= \frac{u(r_jx)}{u(x_j)}.
\]
Denote $\tilde Q_j:=\{x: r_jx\in Q_{1}\}$, $\tilde \Gamma_j:= \{x: r_jx\in \Gamma_{1}\}$, and $\tilde d_j(x)=dist(x,\tilde \Gamma_j)$.
It follows from \eqref{eq:harnack} that for any $R>1$ there exists $C(R)>0$ such that
\[
w_j(x)\le C(R)\tilde d_j(x)  \quad \forall~ \frac1R\le |x|\le R \mbox{ and large }j.
\]
Furthermore, $w_j$ satisfies
\[
-\Delta w_j=n(n-2)(r_j^{\frac{n-2}{2}} u(x_j))^{\frac{4}{n-2}} w_j^{\frac{n+2}{n-2}} \quad \mbox{in }\tilde Q_j,
\]
\[
w_j= 0 \quad \mbox{on }\tilde \Gamma_j \setminus \{0\}.
\]
By the up to boundary estimates for linear elliptic equation, after passing to a subsequence, as $j\to \infty$,
\[
w_j \to w\quad \mbox{in }C^2_{loc}(\bar R^n_+\setminus \{0\}),
\]
and $0<w\in C^2_{loc}(\bar R^n_+\setminus \{0\})$ satisfies
\[
-\Delta w=0\quad \mbox{in }\R^n_+,
\]
\[
w(x',0)=0 \quad \mbox{for all }x'\neq 0.
\]
By Lemma \ref{lem:bocher},
\[
w(x)=a \frac{x_n}{|x|^n}+bx_n,
\]
where $a,b\ge 0$ are constants. By \eqref{eq:criticalpoint} and $w_j(e_n)=1$, we have
\[
0=\frac{\pa }{\pa r}(x_n^{-1} |x|^{\frac{n}{2}} w(x) )\Big|_{x=e_n}=\frac{n}{2}(b-a)
\]
and $a+b=1$. Thus $a=b=\frac{1}{2}$.

By \eqref{eq:ub-1} and \eqref{eq:ub-2}, we have
\[
\lim_{r_j\to 0} P(u, r_j)=0.
\]
It follows from Lemma \ref{lem:pohozaev} that
\[
P(u, r_j)=0 \quad \mbox{for all }j.
\]
On the other hand,
\[
0=P(u,r_j)= P(r_j^{\frac{n-2}{2}} u(r_j x), 1)= P(r_j^{\frac{n-2}{2}} u(x_j) w_j(x),1 ).
\]
Therefore, as $j\to \infty$
\begin{align*}
0&= \int_{\tilde Q_j \cap \pa B_1^+} \frac{n-2}{2} w_j\frac{\pa w_j}{\pa r}-\frac{1}{2} |\nabla w_j|^2+|\frac{\pa w_j}{\pa r}|^2+\frac{(n-2)^2}{8} (r_j^{\frac{n-2}{2}} u(x_j))^{\frac{4}{n-2}}   w_j^{\frac{2n}{n-2}}\, \ud S\\&
\to \int_{\pa'' B_1^+} \frac{n-2}{2} w\frac{\pa w}{\pa r}-\frac{1}{2} |\nabla w|^2+|\frac{\pa w}{\pa r}|^2\,\ud S\\&
= -\frac{1}{4n}|\mathbb{S}^{n-1}|.
\end{align*}
We obtain a contradiction.  The proposition is proved.

\end{proof}

\begin{prop}\label{prop:removable} Let $u\in C^2(\bar Q_1 \setminus \{0\})$ be a nonnegative solution of \eqref{eq:main}. If
\be\label{eq:vanishing}
\lim_{Q_1\ni x\to 0} d(x)^{-1} |x|^{\frac{n}{2}} u(x)=0,
\ee
then $0$ is a removable singular point of $u$.

\end{prop}

Proposition \ref{prop:removable} is included in Theorem 7.1 of \cite{BPV}. We provide

\begin{proof}[Another proof of Proposition \ref{prop:removable}] Since the critical equation is conformally invariant, we may assume $Q_\sigma$ is convex for some small $\sigma>0$, otherwise one may preform a kelvin transform centered in $Q_1$.   For any $0<\mu\le n$, we have
\[
\Delta \frac{x_n}{|x|^\mu}= -\mu(n-\mu) |x|^{-(\mu+2)} x_n.
\]
For any $\va>0$, let
\[
\phi_\va= \al \frac{x_n}{|x|}+\va \frac{x_n}{|x|^{n-1}},
\]
where $\al>0$ is a constant to be fixed. Hence,
\[
(\Delta +(n-1)|x|^{-2}) \phi_\va=0 \quad \mbox{in }B_{1}^+, \quad \phi_\va(x',0)=0 \quad \mbox{for }x'\neq 0.
\]
By \eqref{eq:vanishing}, for any $\delta>0$ there exists $\tau>0$ such that
\[
a(x): =n(n-2) u(x)^{\frac{4}{n-2}} \le \delta |x|^{-2}\quad \mbox{for }|x|\le \tau<\sigma.
\]
Therefore, we have
\[
(\Delta +a(x)) (\phi_\va-u)(x)=-((n-1)|x|^{-2}-a(x)) \phi_\va(x)<0 \quad \mbox{in }Q_\tau.
\]
Choose $\delta<n-1$ small and thus the assumptions in Lemma \ref{lem:mp} are satisfied. Thinks to Lemma \ref{lem:boundary-nonsing}, one can choose  $\al$ such that $\al \frac{x_n}{|x|} \ge u$ on $\pa B_{\tau}^+\cap Q_{\tau}$. Since $Q_\tau$ is convex, $\phi_\va\ge 0=u$ on the bottom boundary of $Q_\tau$.
By \eqref{eq:vanishing}, we have
\[
\liminf_{x\to 0} (\phi_\va-u)(x)\ge 0.
\]
It follows from Lemma \ref{lem:mp} that $u\le \phi_\va$ in $Q_{\tau}$ for all $\va>0$. Sending $\va\to 0$, we have
\[
u\le \al \frac{x_n}{|x|} \quad \forall~x\in Q_{\tau}.
\]
It follows that $0$ is a removable singularity. We complete the proof.

\end{proof}

\section{Asymptotic symmetry and proof of Theorem \ref{thm:1}}
\label{sec:symmetry}

\begin{prop}\label{prop:movingsphere-1} Let $u\in C^2(\bar Q_1 \setminus \{0\})$ be a nonnegative solution of \eqref{eq:main}. Suppose that $\psi$ is concave and $0$ is a non-removable singularity. Then there exists $\va>0$ such that for every $x=(x',x_n)\in \Gamma_1$ with $|x'|<\va$ there holds
\[
u_{x_,\lda}(y)\le u(y) \quad \forall~ 0<\lda<|x'|, ~ y\in Q_{3/4}\setminus B_{\lda}^+(x),
\]
where
\[
u_{x_,\lda}(y):= \left(\frac{\lda}{|y-x|}\right)^{n-2} u(x+\frac{\lda^2 (y-x)}{|y-x|^2}).
\]

\end{prop}

\begin{proof} To present our idea more clear, let us assume $\psi=0$ at the moment.   The proposition is proved as long as the three steps have been through:
\begin{itemize}
\item[(a).] There exists $0<\va<1/10$ such that for every $x=(x',0)$ with $|x'|<\va$
\[
u_{x_,\lda}(y)< u(y) \quad \forall~ 0<\lda<|x'|, ~ y\in\pa'' B_{3/4}^+.
\]
\item[(b).] There exists $0<\lda_1<|x|$ such that
\[
u_{x_,\lda}(y)\le u(y) \quad \forall~ 0<\lda<\lda_1, ~ y\in B_{3/4}^+\setminus B_{\lda}^+(x).
\]
\item[(c).] Let
\[
\bar \lda_x= \sup\{0<\lda<|x|: u_{x_,\mu}(y)\le u(y) \quad \forall~ 0<\mu<\lda, ~ y\in B_{3/4}^+\setminus B_{\mu}^+(x)\}.
\]
Then $\bar \lda_x= |x|$.
\end{itemize}

\emph{Proof of (a).} By Lemma \ref{lem:boundary-sing}, there exists a constant $c_0>0$ such that
\be\label{eq:22lowerbound}
u(y)\ge c_0y_n \quad \mbox{on }B_{3/4}^+.
\ee
For $0<\lda<|x|<\va, ~y\in \pa'' B_{3/4}$, we have
\[
\left|x+\frac{\lda^2 (y-x)}{|y-x|^2}\right| \ge |x|-\frac{20}{13} |x|^2 \ge \frac{11}{13}|x|,
\]
and
\[
\frac{\lda^2 y_n }{|y-x|^2} \le \frac{20}{13} |x|^2 y_n.
\]
Thus
\[
\frac{(x+\frac{\lda^2 (y-x)}{|y-x|^2})_n}{|x+\frac{\lda^2 (y-x)}{|y-x|^2}|}\le \frac{15}{11} |x|<\frac{3}{22}.
\]
It follows from Corollary \ref{cor:partial bound} with $\gamma=\frac{3}{22}$ that
\[
u(x+\frac{\lda^2 (y-x)}{|y-x|^2}) \le C y_n |x|^{-\frac{n-4}{2}}.
\]
Hence
\[
u_{x,\lda}(y)= \left(\frac{\lda}{|y-x|}\right)^{n-2} u(x+\frac{\lda^2 (y-x)}{|y-x|^2}) \le C y_n \lda^{n-2} |x|^{-\frac{n-4}{2}} \le Cy_n  |x|^{\frac{n}{2}} < c_0 y_n,
\]
provided
\[
0<\va<(\frac{c_0}{C})^{\frac{2}{n}}.
\]

\emph{Proof of (b).} For any fixed $0<|x|<\va$, we claim there exist $0<\lda_3<\lda_2<|x|$ such that
\be \label{eq:27}
u_{x,\lda}(y)\le u(y) \quad \forall~ 0<\lda\le \lda_3, ~ y\in B_{3/4}^+(x)\setminus B_{\lda_2}^+(x).
\ee
Indeed, for every $0<\lda_2<|x|$ and every $y\in B_{3/4}^+\setminus  B_{\lda_2}^+(x) $, we have $x+\frac{\lda^2(y-x)}{|y-x|^2}\in B_{\lda_2}^+(x)$. Hence,
\begin{align*}
u_{x,\lda}(y)&= \left(\frac{\lda}{|y-x|}\right)^{n-2} u(x+\frac{\lda^2(y-x)}{|y-x|^2})\\&
\le \left(\frac{\lda}{|y-x|}\right)^{n-2}   \frac{\lda^2 y_n}{|y-x|^2}\sup_{z\in B_{\lda_2}^+(x)} \frac{u(z)}{z_n}\\&
 \le c_0y_n\le u(y)
\end{align*}
where $c_0$ is the constant in \eqref{eq:22lowerbound},
\[
0<\lda_3=  \lda_2 \Big( c_0\Big/ \sup_{z\in B_{\lda_2}^+(x)} \frac{u(z)}{z_n}\Big)^{1/n},
\]
and Lemma \ref{lem:boundary-nonsing} has been used. Therefore, \eqref{eq:27} is confirmed.

We are going to use the \emph{narrow domain technique} to conclude that the remaining case: $u_{x,\lda}\le u$ in $ B_{\lda_2}^+(x)\setminus B_{\lda}^+(x)$.

By \eqref{eq:27},  $u_{x,\lda}\le u$ on $\pa\big( B_{\lda_2}^+(x)\setminus B_{\lda}^+(x)\big)$ for all $0<\lda<\lda_3$. Multiplying both sides of the equation
\[
-\Delta (u_{x,\lda}-u)=n(n-2)(u_{x,\lda}+\tau u)^{\frac{4}{n-2}}(u_{x,\lda}-u) \quad \mbox{in }  B_{\lda_2}^+(x)\setminus B_{\lda}^+(x)
\]
by $(u_{x,\lda}-u)^+$ and integrating by parts, where we used mean value theorem and $0\le \tau=\tau(x)\le 1$, we have, using H\"older inequality,
\begin{align}
&\int_{ B_{\lda_2}^+(x)\setminus B_{\lda}^+(x)} |\nabla(u_{x,\lda}-u)^+|^2 \nonumber \\&=n(n-2) \int_{ B_{\lda_2}^+(x)\setminus B_{\lda}^+(x)}(u_{x,\lda}+\tau u)^{\frac{4}{n-2}}|(u_{x,\lda}-u)^+|^2\nonumber  \\&
\le C(n) \left(\int_{B_{\lda_2}(x)} u^{\frac{2n}{n-2}}\,\ud y\right)^{2/n} \left( \int_{ B_{\lda_2}^+(x)\setminus B_{\lda}^+(x)} |(u_{x,\lda}-u)^+|^{\frac{2n}{n-2}} \right)^{(n-2)/n},
\end{align}
where $C(n)=n(n-2) 2^{\frac{4}{n-2}}$.
Since $(u_{x,\lda}-u)^+ \in H_0^1(B_{\lda_2}^+(x)\setminus B_{\lda}^+(x))$, by Sobolev inequality we have
\[
\int_{ B_{\lda_2}^+(x)\setminus B_{\lda}^+(x)} |\nabla(u_{x,\lda}-u)^+|^2  \ge \frac{1}{S(n)}  \left( \int_{ B_{\lda_2}^+(x)\setminus B_{\lda}^+(x)} |(u_{x,\lda}-u)^+|^{\frac{2n}{n-2}} \right)^{(n-2)/n},
\]
where $S(n)>0$ depends only dimension $n$. Choosing $\lda_2$ small to ensure
\be \label{eq:lda_2}
C(n) \left(\int_{B_{\lda_2}(x)} u^{\frac{2n}{n-2}}\,\ud y\right)^{2/n}  \le \frac{1}{2S(n)} ,
\ee we obtain
\[
 \frac{1}{2S(n)}  \left( \int_{ B_{\lda_2}^+(x)\setminus B_{\lda}^+(x)} |(u_{x,\lda}-u)^+|^{\frac{2n}{n-2}} \right)^{(n-2)/n} \le 0,
\]
which implies $u_{x,\lda}\le u$ in $ B_{\lda_2}^+(x)\setminus B_{\lda}^+(x)$ because $(u_{x,\lda}-u)^+$ is continuous in $B_{\lda_2}^+(x)\setminus B_{\lda}^+(x)$.
Let $\lda_1=\lda_3$ and we complete the proof.

\emph{Proof of (c).}  By the previous step, we see that $\bar \lda_x$ is well defined. If $\bar \lda_x<|x|$, we have $u-u_{x,\bar \lda_x}\ge 0$ in $B_{3/4}^+\setminus B_{\bar \lda_x}(x)$. By item (a) and strong maximum principle $u-u_{x,\bar \lda_x}> 0$ in $B_{3/4}^+\setminus \bar B_{\bar \lda_x}(x)$.
It follows from Lemma \ref{lem:boundary-sing} that for every $r>\bar  \lda_x $ there exists $c_r>0$ such that
\be\label{eq:24lowerbound}
(u-u_{x,\bar \lda_x})(y)\ge c_r y_n \quad  \forall~ |y-x|\ge r, ~y\in B_{3/4}.
\ee
For $y\in B_{3/4}$ with $|y-x|\ge r$ and $\bar \lda_x<\lda<r\le \frac{|x|+\bar \lda_x}{2}$, making use of Lemma \ref{lem:boundary-nonsing} we have
\begin{align}
|u_{x,\bar \lda_x}(y)-u_{x,\lda}(y)| &\le |y-x|^{2-n}|\bar \lda_x^{n-2}-\lda^{n-2}| u(x+\frac{\bar \lda_x^2(y-x)}{|y-x|^2})\nonumber\\& \quad +(\frac{\lda}{|y-x|})^{n-2}|u(x+\frac{\bar \lda_x^2(y-x)}{|y-x|^2})-u(x+\frac{\lda^2(y-x)}{|y-x|^2})|\nonumber \\&
\le C (\lda-\bar \lda_x)^\al y_n, \label{eq:comparison}
\end{align}
where $\al\in (0,1)$ and $C\ge 1$ depend only on $n$ and $\|u\|_{L^\infty(B_{(|x|+\bar \lda_x)/2}(x))}$,    and we have used
\[
u(x+\frac{\lda^2(y-x)}{|y-x|^2})\le C \frac{\lda^2y_n}{|y-x|^2}
\]
and
\begin{align*}
&|u(x+\frac{\bar \lda_x^2(y-x)}{|y-x|^2})-u(x+\frac{\lda^2(y-x)}{|y-x|^2})|\\& = \frac{\bar \lda_x^2 y_n}{|y-x|^2} \left|\frac{u(x+\frac{\bar \lda_x^2(y-x)}{|y-x|^2})}{\frac{\bar \lda_x^2 y_n}{|y-x|^2}}- \frac{u(x+\frac{\lda^2(y-x)}{|y-x|^2})}{\frac{\bar \lda_x^2 y_n}{|y-x|^2}} \right|\\&
\le C y_n \left|\frac{u(x+\frac{\bar \lda_x^2(y-x)}{|y-x|^2})}{\frac{\bar \lda_x^2 y_n}{|y-x|^2}}- \frac{u(x+\frac{\lda^2(y-x)}{|y-x|^2})}{\frac{ \lda^2 y_n}{|y-x|^2}} \right|+Cy_n \frac{\lda^2-\bar \lda_x^2}{\bar \lda_x^2} \\&
\le C(\lda-\bar\lda_x)^{\al} y_n.
\end{align*}
Let $\delta>0$ satisfy
\be\label{eq:delta}
C \delta^\al<\frac{1}{2}c_r, \quad \delta<\frac{|x|-\bar \lda_x}{2}.
\ee
Then by \eqref{eq:24lowerbound} we have for all $\bar \lda_x \le   \lda\le \bar \lda_x+\delta$
\be\label{eq:25lowerbound}
(u-u_{x,\lda})(y)\ge \frac12c_r y_n \quad  \forall~ |y-x|\ge r,~ y\in B_{3/4}.
\ee
This implies that $u-u_{x,\lda}\ge 0$ on $\pa (B_{r}^+(x)\setminus B_{\lda}^+(x))$.  Using \emph{narrow domain technique} as before,  we immediately obtain
\[
u \ge u_{x,\lda} \quad \mbox{in }B_{r}^+(x)\setminus B_{\lda}^+(x),
\] whenever $r$ is chosen such that
\be \label{eq:r}
n(n-2) 2^{\frac{4}{n-2}} \left(\int_{B_{r}^+(x)\setminus B_{\bar\lda_x}^+(x)} u^{\frac{2n}{n-2}}\right)^{2/n}\le
\frac{1}{2S(n)}.
\ee
In conclusion,
\[
u_{x,\lda}(y)\le u(y) \quad \forall~  \bar \lda_x \le \lda\le \bar \lda_x+\delta, ~ y\in B_{3/4}^+\setminus B_{\lda}^+(x).
\]
This contradicts to the definition of $\bar \lda_x$. Hence, $\bar \lda_x =|x|$.

Therefore, we proved Proposition \ref{prop:movingsphere-1} when $\psi=0$.

If $\psi\neq 0$ is concave, we note that for each  $y\in \Gamma_1 \cap B_\lda(x)$, where $x\in \Gamma_1$, $|x'|<\va$ and $\lda<|x'|$, whenever $|y'+ \frac{\lda^2(y'-x')}{|y-x|^2}| \le \frac{3}{4} $ then $y+ \frac{\lda^2(y-x)}{|y-x|^2}\in Q_{3/4}$. Hence, with a little modification  of the above proof for case  $\psi=0$, we complete the proof of Proposition \ref{prop:movingsphere-1}.

\end{proof}

\begin{proof}[Proof of Proposition  \ref{prop:2}] We are going to show that for all $x\in \pa \R^n_+$, $x\neq 0$
there holds
\be\label{eq:goal-1}
u_{x,\lda}(y)\le u(y) \quad \forall~ 0<\lda<|x|, ~ |y-x|\ge \lda.
\ee
The idea is the same as that of the proof of Proposition \ref{prop:movingsphere-1}.

\emph{Step 1.} We prove that \eqref{eq:goal-1} holds for all $0<\lda<\lda_1$ with $\lda_1>0$ small.

Corresponding to the step (a) of the proof of Proposition \ref{prop:movingsphere-1},
by Lemma \ref{lem:exterior} and Lemma \ref{lem:boundary-nonsing} we have for  $0<\lda_3<\lda_2<|x|$
\be\label{eq:30lowerbound}
u(y)\ge \frac{\lda_2^n y_n}{|y-x|^{n}} \inf_{y\in \pa'' B_{\lda_2}^+(x), y_n>0} \frac{u(y)}{y_n}
\ee
and
\[
\inf_{y\in \pa'' B_{\lda_2}^+(x), y_n>0} \frac{u(y)}{y_n}  >0.
\]
It follows from Lemma \ref{lem:boundary-nonsing} that
\begin{align*}
u_{x,\lda}(y)&= \left(\frac{\lda}{|y-x|}\right)^{n-2} u(x+\frac{\lda^2(y-x)}{|y-x|^2})\\&
\le  C\left(\frac{\lda}{|y-x|}\right)^{n-2} \frac{\lda^2 y_n}{|y-x|^2} =C \frac{\lda^n y_n}{|y-x|^n}.
\end{align*}
In view of \eqref{eq:30lowerbound}, by setting $\lda_3\le \lda_2/C^{1/n}$ we showed that
\[
u_{x,\lda}(y)\le u(y) \quad \mbox{for }0<\lda<\lda_3, ~ |y-x|\ge \lda_2.
\]
As in the step (b) of the proof of Proposition \ref{prop:movingsphere-1}, by narrow domain technique we can prove easily that
\[
u_{x,\lda}(y)\le u(y) \quad \forall~ 0<\lda<\lda_3, \lda\le |y-x|\le \lda_2,
\]
where $\lda_2$ is selected to ensure \eqref{eq:lda_2}.

\emph{Step 2.} Define
\[
\bar \lda_x= \sup\{0<\lda<|x|: u_{x_,\mu}(y)\le u(y) \quad \forall~ 0<\mu<\lda, ~ y\in \R^n_+\setminus B_{\mu}^+(x)\}.
\]
By the previous step, $\bar \lda_x>0$ is well defined. We shall prove $\bar \lda_x=|x|$. If not, i.e., $\bar \lda_x<|x|$, we want to show that there exists $0<\delta<\frac{|x|-\bar \lda_x}{2}$ such that \eqref{eq:goal-1} holds for all $0<\lda<\bar \lda_x+\delta$. This obviously contradicts to the definition of $\bar \lda_x$.

By the definition of $\bar \lda_x$, we have $u-u_{x,\bar \lda_x}\ge 0$ in $\R^n_+\setminus B_{\bar\lda_x}(x)$ and thus
\[
-\Delta(u-u_{x,\bar \lda_x})\ge 0.
\]
Since $0$ is a non-removable singularity of $u$, we have $\limsup_{y\to 0}(u-u_{x,\bar \lda_x})(y) =\infty$. By strong maximum principle, we have
\[
u-u_{x,\bar \lda_x}> 0 \quad \mbox{in }\R^n_+\setminus \bar B_{\bar\lda_x}(x).
\]
By Lemma \ref{lem:exterior} and Lemma \ref{lem:boundary-nonsing},
\be\label{eq:31lowerbound}
u(y)\ge \frac{r^n y_n}{|y-x|^{n}} \inf_{y\in \pa'' B_{r}^+(x), y_n>0} \frac{u(y)}{y_n}
\ee
and
\[
\inf_{y\in \pa'' B_{r}^+(x), y_n>0} \frac{u(y)}{y_n}  >0
\]
for every $r>\bar \lda_x$. $r$ will be fixed to ensure \eqref{eq:r} when using the \emph{narrow domain technique}.  Choosing $0<\delta<\frac{|x|-\bar \lda_x}{2}$ sufficiently small ensures that \be\label{eq:comparision-2}
|u_{x,\lda}(y)-u_{x,\bar \lda_x}(y)|\le \frac{1}{2}  \frac{r^n y_n}{|y-x|^{n}} \inf_{y\in \pa'' B_{r}^+(x), y_n>0} \frac{u(y)}{y_n} \quad \forall~y\in \R^n_+, |y|\ge r, \bar \lda_x\le \lda\le \bar \lda_x+\delta.
\ee
Indeed, notice that $x+\frac{\lda^2(y-x)}{|y-x|^2}\in B_{\frac{|x|+\bar \lda_x}{2}}^+(x)$. By  Lemma \ref{lem:boundary-nonsing} and computing as in deriving \eqref{eq:comparison} we have
\begin{align*}
|u_{x,\bar \lda_x}(y)-u_{x,\lda}(y)| &\le |y-x|^{2-n}|\bar \lda_x^{n-2}-\lda^{n-2}| u(x+\frac{\bar \lda_x^2(y-x)}{|y-x|^2})\\& \quad +(\frac{\lda}{|y-x|})^{n-2}|u(x+\frac{\bar \lda_x^2(y-x)}{|y-x|^2})-u(x+\frac{\lda^2(y-x)}{|y-x|^2})|\\&
\le C \delta^\al \frac{y_n}{|y-x|^n},
\end{align*}
where $\al\in (0,1)$ and $C\ge 1$ depend only on $n$ and $\|u\|_{L^\infty(B_{(|x|+\bar \lda_x)/2})}$. Hence, \eqref{eq:comparision-2} holds by setting
\[
C \delta^\al \le \frac{r^n}{2}\inf_{y\in \pa'' B_{r}^+(x), y_n>0} \frac{u(y)}{y_n}.
\]
Hence,
\[
u_{x,\lda}(y)\le u(y) \quad \forall~ y\in \R^n_+, |y|\ge r, \bar \lda_x\le \lda\le \bar \lda_x+\delta<r.
\]
By \emph{narrow domain technique}, the above inequality holds for all $y\in \R^n_+$ with $|y|\ge \lda$. Therefore, step 2 is finished.

Let $e=(e',0)\in \R^n$ be an arbitrary unit vector, $a>0$ constant, and $y\in \R^n_+$ satisfying $ye-a<0$,  \eqref{eq:goal-1} holds for $x= Re$ and $\lda =R-a$: \[
u(y)\ge (\frac{R-a}{|y-Re|})^{n-2}u(x+\frac{(R-a)^2(y-Re)}{|y-Re|^2}).
\]
Sending $R\to \infty$, we have
\[
u(y)\ge u(y-2(y\cdot e-a)e)=u(y'-2(y'\cdot e'-a)e',y_n).
\]
Proposition  \ref{prop:2} follows immediately.

\end{proof}

\begin{prop}\label{prop:sym} Let $u\in C^2(\bar Q_1 \setminus \{0\})$ be a nonnegative solution of \eqref{eq:main}. Then there exists a constant $\bar \delta >0$ depends on the $C^2$ norm of $\psi$ such that for $x\in Q_{1/2}$ with $\frac{d(x)}{|x|}<\gamma<1$ and $x_n>\frac{1}{\delta}\max_{|y'|=|x'|}|\psi(y')|$ we have
\be\label{7'}
u(x',x_n)=\bar u(|x'|,x_n)(1+O(|x|)),
\ee
where $\bar u(x',x_n)=\dashint_{\mathbb{S}^{n-2}}u(|x'|\theta, x_n)\,\ud \theta$ and $O(|x|)\le C(\gamma)|x|$ as $x\to 0$.

\end{prop}

\begin{proof} Suppose first that $\psi$ is concave.
For $r>0$,  let $ x_\al=(x_\al',x_n)$ and $x_\beta=( x'_\beta,x_n)$ be two points in $Q_{\va}$ with $x_n>\max_{|y'|=r}\psi(y')$ such that
\[
u( x_\al)=\max_{|x'|=r} u(x',x_n)\quad \mbox{and} \quad u( x_\beta)=\min_{|x'|=r} u(x',x_n),
\]
where $\va$ is the one in Proposition \ref{prop:movingsphere-1}. Let
\[
x_\gamma'=x'_\al+\frac{\va(x'_\al-x'_\beta)}{4| x'_\al- x'_\beta|}\quad \mbox{and} \quad  x_\gamma=(x'_\gamma,\psi(x_\gamma')).
\]
We want to find $
\tilde x_\beta= (x_\beta', t) $ and $\lda>0$ such that
\[
x_\gamma+\frac{\lda^2 (\tilde x_\beta-x_\gamma)}{|\tilde x_\beta-x_\gamma|^2}=x_\al.
\]
It follows that
\[
t= \frac{4}{\va}(x_n-\psi(x_\gamma'))|x_\al'-x_\beta'|+x_n
\]
and
\[
\lda^2 =\frac{\va (|x_\beta'-x_\gamma'|^2+\va^{-2}(x_n-\psi(x_\gamma'))^2(4|x_\al'-x_\beta'|+\va)^2)}{4|x_\al'-x_\beta'|+\va} .
\]
It is easy to check that there exist positive  constants $\bar r$ and $\bar C(\va)$, depending only on $\va$ and $\psi(x_\gamma')$, such that  that if $|x_n|\le \bar C(\va)\sqrt{r}$ and $r<\bar r$ then $\lda^2<|x_\gamma|^2$.
By Proposition \ref{prop:movingsphere-1}, we have
\[
(\frac{\lda }{|\tilde x_\beta-x_\gamma |})^{n-2} u(x_\al)=u_{x_\gamma, \lda}(\tilde x_\beta)\le u(\tilde x_\beta).
\]

Let $s=\sqrt{r^2+x_n^2}$ and $u_s(x)=s^{\frac{n-2\sigma}{2}} u(s x)$. Then
\be\label{eq:u-s}
-\Delta u_s=n(n-2)u_s^{\frac{n+2}{n-2}}=:V(x) u_s \quad \mbox{in }\tilde Q_2\setminus  \tilde Q_{1/2},\quad u_s=0 \quad \mbox{on }\tilde \Gamma _2\setminus \tilde \Gamma_{1/2},
\ee
where $V(x):= n(n-2) u_s^{\frac{4}{n-2}}$ and $\tilde Q_R$ and $\tilde \Gamma_R$ are the scalings of $Q_{sR}$ and $\Gamma(sR)$. Let $\om_\gamma=\{x\in \tilde Q_{3/2}\setminus  \tilde Q_{3/4}:\frac{ d_{sx}}{s|x|}<\gamma\}$.  By elliptic estimates to the boundary and Lemma \ref{lem:boundary-nonsing}, we have
\[
|\nabla u_s(\frac{1}{s}(x_\beta+\theta  \tilde x_\beta))|\le \frac{ C \sup_{\om_\gamma} u_s}{dist(\frac{1}{s}(x_\beta+\theta  \tilde x_\beta), \tilde \Gamma _2\setminus \tilde \Gamma_{1/2})}\le \frac{C u_s(\frac{1}{s} \tilde x_\beta)}{dist(\frac{1}{s} \tilde x_\beta, \tilde \Gamma _2\setminus \tilde \Gamma_{1/2})}
\]
for every $\theta\in (0,1)$, where $C$ depends only on $n, \psi$ and the constant $C(\frac{\gamma+1}{2})$ in Proposition \ref{lem:2.1}.
Since $|\tilde x_\beta- x_\beta|=\frac{4 (x_n-\psi(x_\gamma'))|x_\al'-x_\beta'|}{\va}$, by mean value theorem  we have
\[
|\frac{u(x_\beta)}{u(\tilde x_\beta)}-1| \le Cr.
\]
Hence,
\begin{align*}
u(x_\al)& \le u(x_\beta) (1+C r)(\frac{|\tilde x_\beta-x_\gamma |}{\lda })^{n-2}\\&=
u(x_\beta)(1+C  r)(\frac{4|x_\al'-x_\beta'|+\va}{\va})^{\frac{n-2}{2}}\\&
=u(x_\beta)(1+O(r)).
\end{align*}
Therefore, the proposition is proved if $\psi $ is concave.

If $\psi$ is not concave, let $B_\rho(\rho e_n)$ be an inner tangential ball of $Q_1$ contacting $Q_1$ at $0$, where $\rho>0$ for $\Gamma_1$ is of $C^2$.
Let
\[
\tilde \phi(y)= \rho e_n+\frac{\rho^2 (y-\rho e_n)}{|y-\rho e_n|^2} , \quad v(y)=(\frac{\rho}{|y-\rho e_n|})^{n-2} u(\rho e_n+\frac{\rho^2 (y-\rho e_n)}{|y-\rho e_n|^2}),
\]
$D=\phi^{-1} Q_1$ and $\Lda= \phi^{-1} \Gamma_1$. Then $0\in \Lda$, $D$ is concave at $0$ and
\[
\Delta v=n(n-2)v^{\frac{n+2}{n-2}} \quad \mbox{in }D, \quad v= 0 \quad \mbox{on }\Lda.
\]
Note that $|x|\le C(\rho )|y|$ for $|y|\le \frac{\rho}{100}$, and
\[
\frac{\rho}{|y-\rho e_n|}=\rho^{-1}|\frac{\rho^2 (y-\rho e_n)}{|y-\rho e_n|}+\rho e_n-\rho e_n|= |\rho^{-1} x-e_n|=1+O(|x|),
\]
where $x=\frac{\rho^2 (y-\rho e_n)}{|y-\rho e_n|}+\rho e_n $.
By what we proved for concave $\psi$, the proposition follows immediately.

\end{proof}

\begin{prop}\label{prop:sym'} Suppose that $u$ is a solution of \eqref{eq:main} and $u(x)\le C d(x) |x|^{-\frac{n}{2}}$. Then there exists a constant $\bar \delta >0$ depends on the $C^2$ norm of $\psi$ such that
\be\label{7''}
u(x',x_n)=\bar u(|x'|,x_n)(1+O(|x|)) \quad \mbox{as }x\to 0 \mbox{ with }x_n>\frac{1}{\delta}\max_{|y'|=|x'|}|\psi(y')|.
\ee

\end{prop}

\begin{proof} By the proof of Proposition \ref{prop:sym}, we only consider concave  $\psi$ and $x_n\ge \frac{1}{\tilde  C}\sqrt{|x'|}$ for some $\tilde  C>0$.

For $r>0$ and $x_n\ge \frac{1}{\tilde  C}\sqrt{r}$,  let $ x_\al=(x_\al',x_n)$ and $x_\beta=( x'_\beta,x_n)$ be two points  such that
\[
u( x_\al)=\max_{|x'|=r} u(x',x_n)\quad \mbox{and} \quad u( x_\beta)=\min_{|x'|=r} u(x',x_n).
\]
Let $u_s$ satisfy \eqref{eq:u-s} with $s=\sqrt{r^2+|x_n|^2}$. By mean value theorem, Harnack inequality and interior estimates, we have
\[
|u_s(\frac{1}{s}x_\al)-u_s(\frac{1}{s}x_\beta)|\le Cu_s(\frac{1}{s}x_\beta) \frac{1}{s}|x_\al-x_\beta|\le Cu_s(\frac{1}{s}x_\beta)s,
\]
where we used $x_n\ge \frac{1}{\tilde  C}\sqrt{r}$.
It follows that
\[
u(x_\al)=u(x_\beta)(1+O(s)).
\]
We complete the proof.

\end{proof}

\begin{proof}[Proof of Theorem \ref{thm:1}] The first part of the theorem follows from Proposition \ref{lem:2.1} and Proposition \ref{prop:sym}.
The second part follows from Proposition \ref{prop:lowbound}, Proposition \ref{prop:removable} and Proposition \ref{prop:sym'}.
\end{proof}

\bigskip

\noindent J. Xiong

\noindent School of Mathematical Sciences, Beijing Normal University\\
Beijing 100875, China\\[1mm]
Email: \textsf{jx@bnu.edu.cn}

\end{document}